\documentclass[12pt]{amsart}

\usepackage{amsmath, amssymb, amsthm, times, amsaddr}
\usepackage[margin=1in]{geometry}

\usepackage{graphicx}
\usepackage[all, cmtip]{xy}
\usepackage[utf8]{inputenc}
\usepackage{hyperref,url}

\theoremstyle{plain}
\newtheorem{theorem}{Theorem}
\newtheorem{lemma}{Lemma}
\newtheorem{proposition}{Proposition}
\newtheorem{corollary}{Corollary}

\theoremstyle{definition}

\theoremstyle{remark}
\newtheorem{remark}{Remark}

\newcommand{\ZZ}{\mathbb{Z}}

\DeclareMathOperator{\Ker}{\mathrm Ker}
\DeclareMathOperator{\Img}{\mathrm Im}
\DeclareMathOperator{\End}{End}
\DeclareMathOperator{\E}{\mathcal E}

\DeclareMathOperator{\charac}{\mathrm char}
\DeclareMathOperator{\tr}{\mathrm tr}

\def\Q{{\mathbb Q}}

\DeclareMathOperator{\U}{\mathbb U}
\DeclareMathOperator{\N}{\mathbb N}
\DeclareMathOperator{\F}{\mathcal F}

\DeclareMathOperator{\SL}{SL}
\DeclareMathOperator{\Frac}{Frac}

\DeclareMathOperator{\Rad}{Rad}

\DeclareMathOperator{\Alg}{\bf Alg}

\DeclareMathOperator{\Ga}{{\bf G}_a}

\DeclareMathOperator{\GL}{GL}

\newcommand{\Gm}{\mathbf{G}_{m}}

\newcommand{\A}{\mathcal{A}}
\newcommand{\pa}{\mathfrak{p}}

\DeclareMathOperator{\set}{\bf Set}
\DeclareMathOperator{\grp}{\bf Grp}
\DeclareMathOperator{\LDAG}{\bf LDAG}
\DeclareMathOperator{\LAG}{\bf LAG}
\DeclareMathOperator{\CE}{CentExt}

\begin{document}

\author{Andrei Minchenko}
\thanks{The author was supported by the ISF grant 756/12}
\address{Weizmann Institute of Science}
\email{an.minchenko@gmail.com}
\title{On central extensions of simple differential algebraic groups}

\maketitle

\begin{abstract}
We consider central extensions $Z\hookrightarrow E\twoheadrightarrow G$ in the category of linear differential algebraic groups.  We show that if $G$ is simple non-commutative and $Z$ is unipotent with the differential type smaller than that of~$G$, then such an extension splits. We also give a construction of central extensions illustrating that the condition on differential types is important for splitting. Our results imply that non-commutative almost simple linear differential algebraic groups, introduced by Cassidy and Singer, are simple.   
\end{abstract}
%\keywords{Differential algebraic groups; Chevalley groups}
%\ams{12H05}{12H20, 13N10, 20G05, 20H20, 34M15}

\section{Introduction}\label{sec:Intro}
This paper was motivated by a recent work of Cassidy and Singer~\cite{JH} concerning the structure of linear differential algebraic groups (LDAGs). In general, it is not true that an LDAG has a finite subnormal series with simple successive quotients. However, if one slightly relaxes the requirement for the quotients, one obtains a Jordan-H{\"o}lder type theorem for LDAGs. Namely, in~\cite{JH}, almost simple LDAGs were introduced (see Section~\ref{sec:ASLDAGs} of this paper), which can be thought of as being minimal central extensions of simple LDAGs by groups of smaller differential type. Cassidy and Singer also state a uniqueness property for the quotients. From this perspective, a better understanding of almost simple LDAGs is crucial for the development of the structure theory of LDAGs, which in turn, is important for creating algorithms that compute Galois groups of parametrized linear differential equations (see, e.~g., \cite{MiOvSi1} and \cite{MiOvSi2}).The papers~\cite{diffreductive}, \cite{MinOvRepSL2}, \cite{GO} also contain some recent results on the structure of LDAGs and their representations. 

Our main result is establishing that all non-commutative almost simple LDAGs are simple (Theorem~\ref{thm:main1}). The classification of simple LDAGs over differentially closed fields (with respect to several commuting derivations) is known due to Cassidy~\cite{CassidyClassification}. In our proof, aside from the facts from differential algebra and the theory of LDAGs, we use the results of Steinberg~\cite{Steinberg} and Matsumoto~\cite{Matsumoto} on abstract central extensions of Chevalley groups and the results of Borel and Tits~\cite{BT} on abstract homomorphisms of Chevalley groups. 

It is natural to ask whether our result generalizes if one considers central extensions of simple LDAGs by an arbitrary LDAG, not necessarily of smaller differential type. For example, is it true that all perfect central extensions of $\SL_2$ in the category of LDAGs are isomorphic to $\SL_2$ (as it is for the category of linear algebraic groups)? As we show in Section~\ref{sec:Construction}, the answer is NO (Corollary~\ref{cor:Chev}). A deeper investigation led us to a description of universal central extensions of non-commutative simple LDAGs. This will be the subject of a subsequent publication.

The paper is organized as follows. In Section~\ref{sec:Prelim}, we recall basic definitions and facts that will be used for the proof of the main Theorem~\ref{thm:main1}. Section~\ref{sec:Main} is devoted to the proof of this theorem. In the final Section~\ref{sec:Construction}, we give a construction for central extensions providing the negative answer to whether the aforementioned generalization of our result holds.

\section{Preliminaries}\label{sec:Prelim}
In this section, we briefly recall basic definitions and facts from differential algebra and differential algebraic geometry. For details, see~\cite{Ritt,KolDAG,Kol,KolConstr,CassidyRep,Cassidy,CassidyBuium,Kaplansky}. For basic definitions and facts from the theory of linear algebraic groups, we refer to \cite{Humphreys} and~\cite{Waterhouse}.

\subsection{Basic definitions and conventions}
We will denote the set of natural numbers $0,1,2,\ldots$ by~$\N$. The symbol $\Delta$ will always stand for a finite (possibly, empty) set of letters $\partial_1,\partial_2,\ldots$ and $m\in\N$ will denote its cardinality. All rings are supposed to be commutative with identity. A derivation of a ring $R$ is an additive map $\partial: R\to R$ satisfying Leibniz rule: 
$$
\partial(ab)=\partial(a)b+a\partial(b)\quad\forall a,b\in R.
$$
A $\Delta$-ring is a ring with the set $\Delta$ of commuting derivations. A mophism of $\Delta$-rings $k\to R$ is a homomorphism of rings that commutes with the action of~$\Delta$. Given such a morphism, we call $R$ a $\Delta$-$k$-algebra. If both $k$ and $R$ are fields, we say that $k\to R$ is a $\Delta$-field extension or $k$ is a $\Delta$-subfield of $R$. A morphism of $\Delta$-$k$-algebras is a homomorphism of $k$-algebras which is a morphism of $\Delta$-rings. 

If $A$ and $B$ are $\Delta$-$k$-algebras, their tensor product $A\otimes_kB$ (as of $k$-algebras) is endowed with a unique structure of a $\Delta$-$k$-algebra such that the inclusion maps $A,B\to A\otimes B$ are morphism of $\Delta$-rings.

The word ``differential" will sometimes substitute the symbol $\Delta$ (and vice versa): for example, a differential ring is the same as a $\Delta$-ring. 

The basic example of a $\Delta$-ring is a ring of differential polynomials. Namely, let $Y$ be a set and $k$ a $\Delta$-ring. Consider the polynomial ring $k[y_{d_1\ldots d_m}]_{y\in Y, d_i\in\N}$. It has the structure of a $\Delta$-$k$-algebra determined uniquely by setting 
$$
\partial_i\left(y_{d_1\ldots d_i\ldots d_m}\right):=y_{d_1\ldots (d_i+1)\ldots d_m}.
$$
This $\Delta$-$k$-algebra is denoted $k\{Y\}$. We identify $y\in Y$ with $y_{0\ldots 0}$. The elements of $k\{Y\}$ may be interpreted as functions $k^Y\to k$ so that the evaluation map $\nu_x:k\{Y\}\to k$ corresponding to a point $x\in k^Y$ is a morphism of $\Delta$-$k$-algebras. The $k$-algebra $k\{Y\}$ is filtered by its subalgebras $k\{Y\}_{\leq s}$, $s\in\N$, generated by all $y_{d_1\ldots d_m}$ with $\sum d_i\leq s$. The differential polynomials from $k\{Y\}_{\leq s}$ are said to be of order not exceeding~$s$. 

The notation $k\{X\}$ will also be used later to denote the ring of differential polynomial functions on a $\Delta$-$k$-algebraic set $X$ (Section~\ref{sec:Dsets}). In order to avoid confusion, we make a convention to use letters $Y$, $y$, $y_1$, $y_2,\ldots$ only for the defined above ring of differential polynomials, e.~g., $k\{y_1,\ldots,y_n\}$.

A $\Delta$-$k$-algebra is called $\Delta$-finitely generated if it is a quotient of a $k\{Y\}$ for some finite~$Y$. An ideal $I$ of a $\Delta$-ring $R$ is called differential of it is stable under the action of~$\Delta$. If a subset $F\subset I$ is not contained in a smaller $\Delta$-ideal, $I$ is said to be $\Delta$-generated by $F$. If there exists such finite~$F$, $I$ is called $\Delta$-finitely generated. 

If $R$ is a domain, we write $\Frac R$ for its field of fractions. We will use bold font to denote categories. Most categories in the text will (hopefully) be clear by their names, e.~g. $\set$, $\grp$, $\Delta$-$k$-$\Alg$, for which we will omit the description.

All differential fields will be supposed to have characteristic~$0$.

\subsection{Differential fields}\label{sec:DiffField}
A $\Delta$-field extension $k\to K$ is called \emph{semi-universal} (over $k)$ if, for every $\Delta$-finitely generated $\Delta$-field extension $k\to k'$, there is a $\Delta$-$k$-homomorphism $k'\to K$. If $K$ is semi-universal over all $\Delta$-finitely generated field extensions $k'\subset K$ of~$k$, it is called \emph{universal over~$k$}. Every $\Delta$-field has a universal (over it) $\Delta$-field extension. A universal over $\Q$ $\Delta$-field is simply called \emph{universal}.  

A $\Delta$-field $K$ is called \emph{differentially closed}, or $\Delta$-closed, if every system $\{P_1=\ldots=P_n=0\}$, $P_i\in K\{y_1,\ldots,y_n\}$, that has a solution in some $\Delta$-field extension of $K$, has one in $K$ too. Equivalently, for every prime $\Delta$-ideal $\pa\subset K\{y_1,\ldots,y_n\}$, there exists an $x\in K^n$ such that $f(x)=0$ for all $f\in\pa$. Note that differentially closed fields are algebraically closed. Universal fields are differentially closed.

We reserve letter $\U$ to denote universal $\Delta$-fields. When we write $k\subset\U$, it means that $k$ is a $\Delta$-subfield of $\U$.

\subsection{$\Delta$-algebraic sets}\label{sec:Dsets}
For a $\Delta$-field $k$ and an $n\in\N$, one defines the \emph{Kolchin topology} on $k^n$ by setting its basic closed sets to be the sets of solutions of systems of differential polynomial equations over $k$ in differential indeterminates $y_1,\ldots,y_n$. In particular, Zariski closed sets are Kolchin closed. Kolchin topology is Noetherian. Moreover, if $k$ is differentially closed, there is a 1-1 correspondence between closed subsets of $k^n$ and radical $\Delta$-ideals in $k\{y_1,\ldots,y_n\}$, which to every closed $X$ assigns the $\Delta$-ideal $I(X)$ of differential polynomials vanishing on $X$. 

An (affine) \emph{$\Delta$-algebraic set} is a Kolchin closed subset  $X\subset\U^n$, where $\U$ is a universal field
%\footnote{Although everything will work for differentially closed $\U$, we demand universality to follow classical definitions and for consistency with references.} 
and $n\in\N$. If $k\subset\U$ and the ideal $I(X)\subset \U\{y_1,\ldots,y_n\}$ is generated by $I(X)\cap k\{y_1,\ldots,y_n\}$, $X$ is said to be defined over~$k$, or a \emph{$\Delta$-$k$-algebraic set}.

Let $X\subset\U^n$ be a $\Delta$-$k$-algebraic set. Restrictions of functions from $k\{y_1,\ldots,y_n\}$ to $X$ form a $\Delta$-$k$-algebra $k\{X\}$ of \emph{differential polynomial} functions on $X$. If $Y\subset\U^r$ is a $\Delta$-$k$-algebraic set, we call a map $f:X\to Y$ a \emph{morphism} if its coordinates are locally given by fractions of differential polynomial functions on~$X$. If the coordinates are differential polynomial functions, we call $f$ a differential polynomial morphism. Note that, unlike the situation in algebraic geometry, morphisms $X\to\U$ are, in general, not differential polynomial. Fortunately, the morphisms of linear $\Delta$-algebraic groups, which we define below, behave as expected.

%If $K\subset\U$ is a $\Delta$-subfield containing~$k$, then a $\Delta$-$k$-algebraic set $X$ may be regarded as a $\Delta$-$K$-algebraic set, which we will denote $X_K$. Similarly, every %morphism $f:X\to Y$ of $\Delta$-$k$-algebraic sets can be considered as a morphism of $\Delta$-$K$-algebraic sets, which we denote $f_K:X_K\to Y_K$. 

The image of a morphism of $\Delta$-$k$-sets contains a Kolchin open subset of its closure.

\subsection{Linear differential algebraic groups (LDAGs)}
A \emph{linear differential algebraic group} over $k\subset\U$ is a group object in the category of $\Delta$-$k$-algebraic sets whose product map $G\times G\to G$ is differential polynomial. A map of LDAGs over $k$ is called a \emph{morphism} if it is a morphism of their underlying $\Delta$-$k$-algebraic sets that respects the product maps. The category of LDAGs over $k$ will be denoted $\LDAG_k(\U)$. Note that the image of a morphism of  LDAGs is closed.

The main example of an LDAG is $\GL_n(\U)$. Here, we identify the underlying space of $\GL_n(\U)$ with the $\Delta$-$k$-algebraic set $\{(A,z)\in\End_{\U}(\U^n)\times \U\ :\ z\det A=1\}$. Due to Cassidy~\cite{CassidyRep}, for every $G\in\LDAG_k(\U)$, there is a differential polynomial monomorphism $G\to\GL_n(\U)$. Moreover, every morphism in $\LDAG_k(\U)$ is differential polynomial.%\footnote{Although Cassidy~\cite{CassidyRep} considered the case $k=\U$, the proofs there extend to the case of a base field.}

There is a functorial point of view on LDAGs which does not refer to a choice of $\U$ nor the embedding into $\U^n$. Namely, let $k$ be a differential field. Let $\LDAG_k$ denote the category whose objects are the representable functors $\Delta$-$k$-$\Alg\to\grp$ represented by reduced $\Delta$-finitely generated $\Delta$-$k$-algebras (which have to be a Hopf algebra) and the morphisms are natural transformations. 

\begin{proposition}\label{prop:equiv}
Let $k$ be a $\Delta$-subfield of a universal $\Delta$-field $\U$. There is an equivalence of categories $\E:\LDAG_k\to\LDAG_k(\U)$.
\end{proposition}
\begin{proof}
Set $\E(G):=G(\U)$ and $\E(G\stackrel{f}{\to} H):=\{g\mapsto g f^*\}$, where $f^*$ is a morphism in $\Delta$-$k$-$\Alg$ corresponding (by Yoneda) to $f$. Then $\E$ is essentially surjective, fathful (since $\U$ is differentially closed), and full (by the aforementioned result of Cassidy).
\end{proof}

Sometimes, we will use terminology defined for $\LDAG_k(\U)$ to name objects in $\LDAG_k$ (or vice versa), which will be justified by Proposition~\ref{prop:equiv}.

\subsection{Linear algebraic groups (LAGs) as LDAGs}
If one considers $\U$ as an algebraically closed field and drops the words ``differential" and ``$\Delta-$" in the definitions above, one gets the definitions of the categories $\LAG_k$ and $\LDAG_k(\U)$. Similarly, we have an equivalence $\E_a\LAG_k\to\LAG_k(\U)$. Since the Kolchin topology is finer than the Zariski one, $\LAG_k(\U)$ is just a subcategory of $\LDAG_k(\U)$. 
The forgetful functor $\Delta$-$k$-$\Alg\to k$-$\Alg$ has a left adjoint, which is a generalization of the above construction of $k\{Y\}$ from $k[Y]$  (see, e.~g., \cite{Gillet2} for details). This induces the functor $\F:\LAG_k\to\LDAG_k$  so that one has the commutative diagram:
\begin{equation}\label{eq:diagram}
\xymatrixcolsep{4pc}
\xymatrix{
\LDAG_k \ar[r]^-{\E} & \LDAG_k(\U)\\
\LAG_k \ar[u]^{\F}\ar[r]_{\E_a} & \LAG_k(\U) \ar[u]_{\text{inclusion}}
}
\end{equation}
In particular, one can identify $\LAG_k$ with a subcategory of $\LDAG_k$, which we will do. Note also that, if $k\to K$ is a $\Delta$-field extension, $\F$ commutes with the base change functors $\LAG_k\to\LAG_K$ and $\LDAG_k\to\LDAG_K$. Moreover, if $K$ is a $\Delta$-subfield of $\U$, the whole diagram~\eqref{eq:diagram} commutes with the base change $k\to K$.

\subsection{Simple LDAGs}
For the following, recall some terminology. 
For a ring $R$ and a set $S$ of its derivations, define
$$
R^S:=\{r\in R\ :\ s(r)=0\quad\forall\ s\in S\}.
$$
Then $R^S$ is a subring of $R$. If $R$ is an algebraically closed field of characteristic~$0$, so is $R^S$.

%Let $k$ be a $\Delta$-field. Denote $k\Delta$ the $m$-dimensional $k$-space of $k$-linear combinations of elements of $\Delta$. Define a bi-additive bracket on $k\Delta$ by
%$$
%[a_1\partial_1,a_2\partial_2]:=a_1\partial_1(a_2)\partial_2-a_2\partial_2(a_1)\partial_1.
%$$
%A subspace of $k\Delta$ invariant under the bracket is called a \emph{Lie subspace}.

%Let $S\subset k\Delta$. The functor $\Delta$-$k$-$\Alg\to\Delta$-$k$-$\Alg$ that sends $A$ to $A/[SA]$, where $[SA]\subset A$ is the $\Delta$-ideal $\Delta$-generated by $SA\subset A$, gives rise to a representable functor 
%$$
%\F_S:\LDAG_k\to\LDAG_k,\qquad \left(\F_S(G)\right)(R)=R^S.
%$$
%In particular, $\F=\F_{\emptyset}\mid_{\LAG_k}$. We will write $$\F_S(G)=:G^S.$$

Let $k_0\subset k$ be a subfield of $\Delta$-constants and let $S\subset k\Delta$. Consider the functor $\Delta$-$k_0$-$\Alg\to\Delta$-$k$-$\Alg$ that sends $A$ to $A_k/[SA]$, where $A_k:=A\otimes k$ and $[SA]\subset A_k$ is the $\Delta$-ideal $\Delta$-generated by $SA\subset A_k$. It gives rise to the functor 
$$
\LDAG_{k_0}\to\LDAG_k,\qquad G\mapsto G_k^S,\qquad\text{where }\left(G_k^S\right)(R):=G(R^S).
$$
Sometimes, when $k$ is clear from the context, we will write $G^S$ instead of $G_k^S$.   

A connected group $G\in\LDAG_k$ is called \emph{simple} if the kernel of every non-trivial epimorphism $G\to H$ in $\LDAG_k$ is finite.

\begin{theorem}[{\cite[Theorem 17, p.~231]{CassidyClassification}}]\label{thm:Cassidy}
For every simple non-commutative $G\in\LDAG_{\U}$, there exists a simple split $H\in\LAG_{\Q}$ and a subset $S\subset \U\Delta$ such that $G=H_{\U}^S$.
%\footnote{In fact, it follows more from~\cite{CassidyClassification}. Namely, every non-commutative $H^S$ is simple if $H\in\LAG_{\U}$ is, and $H_1^{S_1}\simeq H_2^{S_2}$ if and only if $H_1\simeq H_2$ in $\LAG_{\U}$ and $S_1$ and $S_2$ generate the same Lie subspace in $\U\Delta$.}
\end{theorem}

\subsection{Differential type}\label{sec:DiffType}
Let $X\subset\U^n$ be an irreducible $\Delta$-$k$-algebraic set. For all $s\in\ZZ_{\geq 0}$, denote $k\{X\}_{\leq s}$ the image of $k\{y_1,\ldots,y_n\}_{\leq s}$ under the restriction map $k\{ y_1,\ldots,y_n\}\to k\{X\}$. The function $$\N\ni s\mapsto\mathrm{tr.deg}_{k}\Frac k\{X\}_{\leq s}$$ coincides, for all sufficiently large~$s$, with the values of a polynomial, which we denote $\omega_X(s)$ (see~\cite[Chapter 0.3]{KolDAG} and~\cite[Chapter II.12]{Kol} for details). If $X$ is infinite, we set $\tau(X):=\deg\omega_X$, the degree of $\omega_X$. If $X$ is a point, there is a convention that $\tau(X)=-1$. The number $\tau(X)$ is called the \emph{$\Delta$-type} of~$X$. For reducible $X$, the $\Delta$-type is defined to be the maximum among the $\Delta$-types of its irreducible components. One has $-1\leq \tau(X)\leq m$. Moreover, $\tau(X)$ is a birational invariant of~$X$. In particular, it is defined for all $G\in\LDAG_k$. If $X$ is infinite and Zariski closed, $\tau(X)=m$.

The $\Delta$-type has the following properties (\cite[Chapter IV.4]{KolDAG}, \cite[Corollaries 2.2 and 2.4]{JH}):
\begin{enumerate}
\item If $\mu:X\to X'$ is an injective (resp., dominant) morphism of $\Delta$-$k$-algebraic sets, then $\tau(X)\leq\tau(X')$ (resp., $\tau(X)\geq\tau(X')$).
\item For every exact sequence $G'\hookrightarrow G\twoheadrightarrow G''$ in $\LDAG_k$, $\tau(G)=\max\{\tau(G'), \tau(G'')\}$.
\end{enumerate}
By an exact sequence above, we mean that $G'$ is the kernel of the epimorphism $G\to G''$.

\subsection{Strongly connected and almost simple LDAGs}\label{sec:ASLDAGs}
For every $G\in\LDAG_k(\U)$, there exists a unique minimal normal closed subgroup $G_0\subset G$ (over $k$) such that $\tau(G/G_0)<\tau(G)$ (\cite[Section 2.2.1]{JH}).  It is called a \emph{strong identity component} of $G$. If $G=G_0$, $G$ is called \emph{strongly connected}. Strongly connected LDAGs are connected and, moreover, $(G_0)_0=G_0$. For every morphism of LDAGs $f: G\to H$, one has $f(G_0)\subset H_0$.

An infinite $G\in\LDAG_k$ is called \emph{almost simple} if, for all exact sequences $G'\hookrightarrow G\twoheadrightarrow G''$ in $\LDAG_k$, $\tau(G')<\tau(G'')$). Almost simple LDAGs are strongly connected. Every exact sequence $G'\hookrightarrow G\twoheadrightarrow G''$ with almost simple $G$ and $G''\neq\{e\}$ is central, that is, $G'\subset Z(G)$ (\cite[Corollary 2.14]{JH}).

For a $G\in\LDAG_k(\U)$ there exists a unique maximal normal closed subgroup $[G,G]\subset G$ (over $k$) such that $G/[G,G]$ is commutative. We call $G$ \emph{perfect} if $[G,G]=G$. Note that perfectness of $G$ does not imply perfectness of the abstract group $G(\U)$. If is $G\in\LDAG_k$ is non-commutative almost simple, then it is perfect~\cite[Proposition 3.4]{JH} and $Z(G)^\circ$ is unipotent~\cite[Proposition 3.5]{JH}. Moreover, $G/Z(G)^\circ$ is a simple LDAG (over $k$)~\cite[Corollary 2.15]{JH}.

For a $G\in\LDAG_k(\U)$, the radical $\Rad G$ of $G$ is a maximal connected solvable subgroup of $G$. As for algebraic groups, one shows that $\Rad G$ is unique. Moreover, it is defined over~$k$, which follows from~\cite[Corollary 2, p.~77]{KolDAG}. The LDAG $G$ is called semisimple if $\Rad G=\{e\}$ (equivalently, $G$ does not contain infinite commutative subgroups). It follows that semisimplicity is preserved under base field extensions. 

\subsection{The group $H^2(G,A)$}\label{sec:CE}
 Recall the group cohomology~\cite{Brown}. Let $G,A$ be groups, where $A$ is abelian, and let $\alpha: G\times A\to A$ be an action of $G$ on $A$. For an $n\in\N$, set $G^0=\{e\}$ and 
$$
C^n(G,A):=\{\text{all functions} \ G^n\to A\}.
$$ 
These abelian groups are cochains in the complex with the differentials $$d^n:C^n(G,A)\to C^{n+1}(G,A),\ n\in\N,$$ defined by
\begin{align*}
(d^nf):=&\alpha(g_0,f_1(g_1,\ldots,g_n))+\sum_{i=1}^n(-1)^nf_i(g_{0}g_{1},g_2,\ldots,g_n)+f(g_0,g_1g_2,g_3,\ldots,g_n)-\ldots\\
& \ldots+(-1)^{n}f(g_0,g_1,g_2,\ldots,g_{n-1}g_n)+(-1)^{n+1}f(g_0,g_1,g_2,\ldots,g_{n-1}).
\end{align*}
One defines the group of $n$-cocycles by $Z^n(G,A):=\Ker d^n$ and $n$-coboundaries by $B^n(G,A):=\mathrm{Im}\,d^{n-1}$, where $\mathrm{Im}\,d^{-1}:=\{0\}$. The $n^{\text{th}}$ cohomology group is $H^n(G,A):=Z^n(G,A)/B^n(G,A)$. It is a bifunctor, covariant in the second argument and contravariant in the first one. 
One has the cup product 
$$
\cup: H^i(G,A)\otimes_{\ZZ}H^j(G,A)\to H^{i+j}(G,A\otimes_{\ZZ} A),
$$
which, on the level of cochains, is defined by
$$
f\cup f'(g_1,\ldots, g_{i+j})=f(g_1,\ldots, g_i)\otimes \alpha(g_1\ldots g_i, f'(g_{i+1},\ldots, g_{i+j})). 
$$
Note that
\begin{align*}
Z^0(G,A)&=H^0(G,A)=A^G,\\
Z^{1}(G,A)&=\{f:G\to A\ \mid\ f(gh)=f(g)+\alpha(g,f(h))\}.
\end{align*}
Suppose now that the action $\alpha$ is trivial, that is, $\alpha(g,a)=a\ $ for all $g\in G$, $a\in A$. Then $H^1(G,A)=Z^1(G,A)$ is the group of homomorphisms $G\to A$. If $c\in Z^2(G,A)$, then $E_c:=G\times A$ is endowed with the group structure defined by
$$
(g_1,a_1)(g_2,a_2):=(g_1g_2,c(g_1,g_2)+a_1+a_2).
$$
The projection $E_c\to G$ is a central extension and has kernel~$A$. Recall that a central extension of $G$ by $A$ is a short exact sequence 
$$
A\hookrightarrow E\stackrel{\pi}{\twoheadrightarrow} G, 
$$
where $A$ embeds into $Z(E)$ (we identify $A$ with its image in~$E$). We will call it also $E$, assuming the rest of the data is given. A morphism $E_1\to E_2$ of two such extensions is given by a commutative diagram
$$
\xymatrix{
&E_1 \ar[d] \ar[rd]^{\pi_1}\\
A \ar[r]\ar[ru] & E_2 \ar[r]_{\pi_2} & G}.
$$
The set of central extensions of $G$ by $A$, where the isomorphic ones are identified, is denoted $\CE(G,A)$. It has an abelian group structure (determined by the Baer sum).
A central extension of $G$ by $A$ is called trivial, or splitting, if it is isomorphic to $A\to A\times G\to G$, where the first map is $a\mapsto(a,1)$ and the second is the projection $(a,g)\mapsto g$. Such extensions corresponds to the zero element of $\CE(G,A)$. In fact, the construction above gives a homomorphism of abelian groups
$$
\gamma_{G,A}:H^2(G,A)\to\CE(G,A).
$$
Moreover, $\gamma_{G,A}$ is an isomorphism. One can define the map (of sets) $\gamma_{G,A}$ in the context of an arbitrary symmetric monoidal category, where $G$ and $A$ are group objects. However, it does not have to be surjective nor injective.

\section{Main Result}\label{sec:Main}
The goal of this section is to prove the following
\begin{theorem}\label{thm:main}
Every central extension 
\begin{equation}\label{eq:ce}
A\hookrightarrow E\stackrel{\pi}{\twoheadrightarrow} G
\end{equation}
in $\LDAG_k$ splits if $G$ is simple, $A$ is unipotent and $\tau(A)<\tau(G)$.
\end{theorem}

As a corollary, by Section~\ref{sec:ASLDAGs} (in particular, that almost simple LDAGs are perfect) we obtain
\begin{theorem}\label{thm:main1}
Non-commutative almost simple LDAGs are simple.
\end{theorem}

We start with the Central Lemma, which concerns set-theoretic representations of a group of rational points of $\SL_2$.  In Section~\ref{sec:proof}, we prove Theorem~\ref{thm:main}. If $A$ and $B$ are subgroups of a group $C$ such that the product map $A\times B\to C$ is surjective, we write $C=AB=A\cdot B$.

\subsection{Central Lemma}
The following lemma will be an essential step in the proof of the main theorem. We state the lemma in a greater generality than we need as the proof works for that. Recall that, for a subgroup $G\subset\GL_n(k)$, we denote $\overline{G}$ its Zariski closure in $\GL_n(k)$.
\begin{lemma}\label{lem:FiniteCenter}
Let $k\subset K$ be fields of characteristic different from $2$, and 
$$
\varrho: \SL_2(k)\to\GL_n(K)
$$ 
a homomorphism of abstract groups. Then $Z\left(\overline{\Img\varrho}\right)$ is finite.
\end{lemma} 

\begin{proof}
Set $G:=\overline{\Img\varrho}\subset\GL_n(L)$. Let $H,B,U,U_-\subset G$ denote the images via $\varrho$ of, respectively, the subgroups of diagonal, upper-triangular,  unipotent upper-triangular and unipotent lower-triangular matrices of $\SL_2(k)$. We have $B=HU=UH$. Hence, by~\cite[Theorem 4.3(b)]{Waterhouse} and~\cite[Corollary 7.4]{Humphreys},
$$
\overline{B}=\overline{H}\cdot\overline{U}=\overline{U}\cdot\overline{H.}
$$
By~\cite[Proposition 7.2(i)]{BT}, $\overline{G}$ is connected and its subset
$$
O=\overline{U_-}\cdot\overline{H}\cdot\overline{U}=\overline{U_-}\cdot\overline{B}\subset\overline{G}
$$ 
is open. Let $s\in G$ be an element acting on $H$ (via conjugation) by inversion. Then $s^2\in H$ and $\overline{U}_-=s\overline{U}s^{-1}$. Since $sO=\overline{U}s\overline{B}\subset\overline{G}$ is open, \cite[Lemma 7.4]{Humphreys} implies $sO\cdot sO=G$. Hence,
$$
\overline{U}s\overline{B}s\overline{B}=\overline{G}.
$$
In particular, every element of $\overline{G}$ is conjugate to an element of
$$
s\overline{B}s\overline{B}=s^{-1}\overline{B}s\overline{B}=\overline{U_-}\cdot\overline{H}\cdot\overline{U}.
$$
We conclude
$$
Z(\overline{G})\subset\overline{U_-}\cdot\overline{H}\cdot\overline{U}.
$$
Now, let
$$
z=uhv\in Z,\quad u\in \overline{U_-}, v\in\overline{U}, h\in \overline{H}.
$$
For every $x\in H$, we have
$$
uhv=z=xzx^{-1}=(xux^{-1})h(xvx^{-1}),
$$
whence
$$
\overline{U_-}\ni(u^{-1},x)=h(v,x)h^{-1}\in\overline{U}.
$$
 Note that every element of $\overline{U}\cap\overline{U_-}$ commutes, element-wise, with $U$ and $U_-$, hence with all~$G$. On the other hand, since $\charac k\neq 2$, there exists a diagonal matrix $y\in\SL_2(k)$ commuting with none of non-trivial unipotent matrices. Hence, by~\cite[Proposition 7.1(iii)]{BT}, there are no non-trivial elements of $\overline{U}$ commuting with $x:=\varrho(y)\in H$, which implies first $\overline{U}\cap\overline{U_-}=\{1\}$ and then $u=v=1$. Therefore, 
$$
Z(\overline{G})\subset \overline{H}.
$$
Since $s$ acts on $H$ by inversion, it acts on $\overline{H}$ by inversion too. This implies that all elements of $Z(\overline{G})$ coincide with their inverses. Let $L$ stand for the algebraic closure of~$K$. Since $\charac L\neq 2$, $\GL_n(L)$ does not contain unipotent matrices of order~$2$. Therefore, $Z(\overline{G})$ is conjugate in $\GL_n(L)$ to a subgroup of diagonal matrices (see, e.~g.,~\cite[Section 15]{Humphreys}). Now, the condition $z^2=1$ for all $z\in Z(\overline{G})$ implies finiteness of $Z(\overline{G})$.
\end{proof}

\begin{remark}\label{rem:char}
 The assumption $\charac k\neq 2$ in Lemma~\ref{lem:FiniteCenter} cannot be omitted, as the following example illustrates. Suppose, $k=\mathbb{F}_2(t)$, where $\mathbb{F}_2$ is a field with two elements, and $K$ is the algebraic closure of~$k$. Let $\partial$ denote the derivation of $k$ defined by $\partial(t)=1$. One can check that the map
$$
\begin{pmatrix}a&b\\ c&d\end{pmatrix}
\mapsto\begin{pmatrix}
1& ac&bd&\partial(a)d+\partial(b)c\\
0&a^2&b^2&\partial(ab)\\
0&c^2&d^2&\partial(cd)\\
0&0&0&1
\end{pmatrix}
$$
defines a representation $\varrho:\SL_2(k)\to\GL_4(K)$. The image of a diagonal matrix has the unipotent part centralizing $\Img\varrho$. Moreover, one can also see that these unipotent parts form an infinite group isomorphic to $k^\times/k_0^{\times}$, where $k_0\subset k$ is the subfield of squares. Hence, $Z(\overline{\Img\varrho})$ is infinite.
\end{remark}

\begin{remark}
In the statement of Lemma~\ref{lem:FiniteCenter}, if one replaces $\SL_2(k)$ by a group of $k$-rational points of an arbitrary split semisimple algebraic group $G$ over~$\ZZ$ and the requirement $\charac k\neq 2$ --- by $\charac k\nmid|W|$, where $W$ is the Weyl group of $G$, one still gets a correct statement, which can be proved using a slight modification of the given proof.
\end{remark}

\subsection{Proof of Theorem~\ref{thm:main}}\label{sec:proof}
Let us first reduce the problem to the case of universal~$k$. Let $k\subset\U$ be a $\Delta$-field extension, where $\U$ is universal. Changing the base to $\U$, we still have a central extension
$$
A_{\U}\hookrightarrow E_{\U}\stackrel{\pi_{\U}}{\twoheadrightarrow} G_{\U}
$$
with the same properties except $G_{\U}$ is now semisimple. Then $G_{\U}$ is an almost direct product of normal simple subgroups~\cite[Theorem 15, p.~227]{CassidyClassification}. These subgroups have the same differential type, since otherwise the strong identity component of $G_{\U}$ would be a proper normal subgroup defined over $k$, which is impossible due to simplicity of $G$. If $\pi_{\U}$ splits over each of these subgroups, it splits over $G_{\U}$.

Suppose now that $k$ is universal. By Theorem~\ref{thm:Cassidy}, there exist a simple split $H\in\LAG_{\Q}$ and an $S\subset k\Delta$ such that $G=H_{k}^S$. If $\widehat{H}\to H$ is a simply connected cover \cite[Section 31.1]{Humphreys}, the pull-back $\widehat{\pi}$ of the induced morphism $\widehat{H}^S\to H^S$ and $\pi$ can be included in the central extension
$$
A\hookrightarrow \widehat{E}\stackrel{\widehat\pi}{\twoheadrightarrow} \widehat{H}_k^S,
$$
which reduces the situation to the case of simply connected $H$, since if $\widehat\pi$ splits, so does $\pi$.

Suppose that $H$ is simply connected. We fix a maximal torus in $H$ and the corresponding root system $\Sigma$. To every $\alpha\in\Sigma$ there corresponds a 3-dimensional root subgroup $H_{\alpha}\subset H$ (isomorphic to $\SL_2$) and $H$ is generated by $H_{\alpha}$, $\alpha\in\Sigma$ \cite[Chapter X]{Humphreys}. We will need the following fact. 
\begin{proposition}\label{prop:Steinberg}
Let $\alpha\in \Sigma$ be a long root and let $T\subset H_\alpha$ be a (1-dimensional) maximal torus.
Let $K$ be an algebraically closed field. If $C\hookrightarrow B\stackrel{\mu}{\twoheadrightarrow} H(K)$ is a central extension of abstract groups such that $\mu^{-1}(T(K))$ is commutative, then $\mu$ splits.
\end{proposition}
\begin{proof}
The restriction homomorphism $\nu:H^2(H(K),A)\to H^2(T(K),A)$ is an embedding: see, e.~g., Theorem 5.10 of~\cite{Matsumoto} and the discussion above it. The fact that $\nu$ maps the class of $\mu$ to $1$ follows from algebraic closedness of $K$ (it suffices, in fact, that all elements of $K$ are squares) and~\cite[Proposition 5.7(b,d)]{Matsumoto}.
\end{proof}
\begin{remark}
Proposition~\ref{prop:Steinberg} also follows from~\cite[Lemma 39(c), p.~70, and Theorem~12, p.~86]{Steinberg}.
\end{remark}

We will aplly Proposition~\ref{prop:Steinberg} for $K:=k^S$, $C:=A(k)$, $B:=E(k)$, $H:=G(k)=H(K)$, $\mu:=\pi(k)$.  Propositions~\ref{prop:Comm} and~\ref{prop:StrConn} will show that the hypothesis of Proposition~\ref{prop:Steinberg} is satisfied, hence $\pi(k)$ splits.

\begin{proposition}\label{prop:Comm}
Let
$
C\hookrightarrow B\stackrel{\mu}{\twoheadrightarrow} T
$
be a central extension in $\LDAG_k$. Suppose that $T$ is commutative strongly connected and $\tau(C)<\tau(T)$. Then $B$ is commutative.
\end{proposition}
\begin{proof}
Denote $B_0$ the strong identity component of $B$. Since $\mu$ induces an epimorphism $B/B_0\to T/\mu(B_0)$, 
$$
\tau(T/\mu(B_0))\leq \tau(B/B_0)<\tau(B)=\tau(T),
$$
which implies $\pi(B_0)=T$ by the strong connectedness of $T$. By $\tau(C)<\tau(T)$ and~\cite[Corollary 2.22]{JH}, $B_0$ and $T$ are either both commutative or both not. Hence, $B_0$ is commutative. Therefore, so is $B=CB_0$.
\end{proof}

Denote $\Gm$ and $\Ga$ the LAGs that take the multiplicative and the additive group of a ring, respectively.
\begin{proposition}\label{prop:StrConn}
The LDAG $T:=(\Gm)^S$ is strongly connected.
\end{proposition}
\begin{proof}
It is straightforward to verify that $T$ is connected. Now, let $\varphi: T\to T'$ be an epimorphism in $\LDAG_k$ with $\tau(T')<\tau(T)$. We need to show that $T'$ is trivial. In the rest, we will use properties of tangent spaces of LDAGs --- see \cite[Chapter III]{Cassidy} for details. The epimorphism $\varphi$ induces the epimorphism of tangent spaces $$d_e\varphi:\Ga^S\simeq T_e(T)\to T_e(T').$$ It suffices to show that $d_e\varphi$ is trivial. Since $\tau(L)=\tau(T_eL)$ for all LDAGs $L$, the $\Delta$-type of $T_e(T')$ is less than that of $\Ga^S$. 

Therefore, $d_e\varphi$ would be trivial if $(\Ga)^S$ equals its strong identity component~$D$. Recall that all endomorphisms of $(\Ga)^S$ preserve $D$. Let $a\in\Ga^S(k)=\Ga(k^S)\subset k$ and show that $a\in D(k)$. Since $D$ is nontrivial, there exists a nonzero $b\in\Ga(k^S)$. Then the endomorphism $x\mapsto \frac{a}{b}x$ of $\Ga^S(k)$ sends $b$ to $a$. We conclude $a\in D(k)$, which finishes the proof. 
\end{proof}

We obtain that $\pi(k):E(k)\to H^S(k)=H(K)$ splits as a morphism of abstract groups. Denote $\sigma: H(K)\to E(k)$ the corresponding splitting. It remains to show that $\sigma$ is a morphism of LDAGs.

Since $H(K)$, as a LAG over $K$, is generated by its 3-dimensional root subgroups, \cite[Proposition 7.5]{Humphreys} implies existence of 3-dimensional root subgroups $H_i\subset H$, $1\leq i\leq n$, such that
$$
H(K)=H_1(K)H_2(K)\cdots H_n(K).
$$ 
Let $\widetilde{H_i}$ denote the Kolchin closure of $\sigma(H_i(K))$ in $E(k)$, $1\leq i\leq n$. We have $\widetilde{H_i}\subset\sigma(H_i(K))A$. Since Zariski closed sets are Kolchin closed, Lemma~\ref{lem:FiniteCenter} implies that $\widetilde{H_i}\cap A$ is finite, hence, trivial, for $A$ is unipotent. Therefore, $\widetilde{H_i}=\sigma(H_i(K))$ and
$$
\widetilde{H_1}\widetilde{H_2}\cdots \widetilde{H_n}=\sigma(H(K)).
$$ 
On the other hand, as an image of a morphism, this product contains an open subset of the Kolchin closure of $\sigma(H(K))$ in $E(k)$. Similar to the case of algebraic groups~\cite[Lemma 7.4]{Humphreys}, this implies that $\sigma(H(K))$ is a Kolchin closed subgroup of $E(k)$. Then the restriction of $\pi(k)$ to $\sigma(H(K))$ is a bijective homomorphism of LDAGs. Then, by \cite[Proposition 8, p.909]{Cassidy}, $\sigma$ is a morphism of LDAGs, which completes the proof.

\begin{remark}\label{rem:Freitag}
James Freitag has let us know about the paper~\cite{Freitag}, where he shows, using both model theory and differential algebra, that the group of $\U$-rational points of a non-commutative almost simple LDAG is perfect as an abstract group. Having such a result, one would no longer need Lemma~\ref{lem:FiniteCenter} to prove Theorem~\ref{thm:main} for the base field $k=\U$. Indeed, then $\sigma:H(K)\to E(k)$ from the proof is automatically surjective, hence an isomorphism of LDAGs.  
\end{remark}

\begin{remark}\label{rem:Pillay}
One could consider a similar problem for differential algebraic groups (DAGs) --- see~\cite{JH} and references there for the definitions. (In this general situation, all objects are  defined as sets of points in some $\U^n$.) If $|\Delta|=1$, one can show that all non-commutative almost simple DAGs are linear, hence, simple by Theorem~\ref{thm:main}. Namely, let $G$ be such a group. By~\cite[Corollary 4.8]{PillayQuestions}, there exists a normal linear $\Delta$-subgroup $N\subset G$ such that the quotient $G/N$ is a Kolchin closed subgroup of an Abelian variety. Since $G$ is almost simple, $N\subset Z(G)$. It follows that $A:=G/Z(G)$ is a Kolchin closed subgroup of some Abelian variety. Moreover, $A$ is connected and is simple as an abstract group. By~\cite[Lemma 4.2]{Pillay2}, $A$ is either trivial or has torsion. Hence, $A$ is trivial. 

As was kindly communicated to the author by Anand Pillay, the statements which we have referred to in this argument extend to the case of an arbitrary~$\Delta$. Hence, all non-commutative almost simple DAGs are simple. 
\end{remark}

\section{A construction of central extensions of simple LDAGs}\label{sec:Construction}
This section is devoted to illustrate that Theorem~\ref{thm:main} does not generalize if one drops the condition $\tau(A)<\tau(G)$ (Corollary~\ref{cor:Chev}). The construction we describe below is also supposed to be used in a subsequent publication, where we intend to describe universal central extensions of simple objects in the category $\LDAG_k$.

Recall that all considered fields are of characteristic~0, although the results of this section extend to the case of positive characteristic with mild restrictions.
\subsection{An abstract construction}\label{sec:absconstr}
Let $G$ be an abstract group, $K$ be a field, and $V$ be a $KG$-module of dimension $n$ over $K$. Then $G$ acts on $\End(V)$ by conjugation. Moreover, we have the $G$-equivariant linear map
$$
\End V\otimes_K\End V\to K,\quad A\otimes B\mapsto \tr(AB).
$$
In composition with the cup-product
$$
Z^1(G,\End(V))\otimes_KZ^1(G,\End(V))\to Z^2(G,\End(V)\otimes_K\End(V)),
$$
this gives us the map
$$
\alpha: Z^1(G,\End(V))\otimes_KZ^1(G,\End(V))\to Z^2(G,K).
$$

Recall that the isomorphism classes of central extensions of $G$ by $K$ correspond to elements of $H^2(G, K)$. For a cocycle $c$, we denote $\bar{c}$ the corresponding class in cohomology, and we write $\bar{\alpha}(c_1\otimes c_2)$ instead of $\overline{\alpha(c_1\otimes c_2)}$. We want to pick two 1-cocycles $c_1,c_2:G\to\End(V)$ such that $\bar{\alpha}(c_1\otimes c_2)\neq 0$. Note the formula
\begin{equation}\label{eq:alpha}
\alpha(c_1\otimes c_2)(g,h)=\tr(c_1(g)gc_2(h)g^{-1})=-\tr(c_1(g^{-1})c_2(h)).
\end{equation}

Every derivation $\partial: K\to K$ determines an element in $H^1(G,\End(V))$ as follows. Choose a $K$-basis in $V$ thus identifying $V$ with $K^n$. The action of $G$ on $V$ yields the homomorphism $G\to\GL_n(K)$. One can verify that the map
$$
c_\partial: \GL_n(K)\to \End(K^n),\qquad A\mapsto\partial(A)A^{-1}
$$
determines an element of $Z^1(G,\End(K^n))$, where $\partial(A)\in\End(K^n)$ is obtained by applying $\partial$ to $A$ entry-wise. A different choice of the basis would give the same element up to a shift by a 1-coboundary.

\begin{proposition}\label{prop:CupProduct}
Let $G$, $K$ and $V$ be as above. Suppose that
\begin{enumerate}
\item[(1)] there given two $K$-linearly independent derivations $\partial_1,\partial_2: K\to K$ and
\item[(2)] there is a homomorphism $\nu:K^\times\to G$, a basis $\{e_1,\ldots,e_n\}$ of $V$, and integers $d_1,\ldots, d_n$,  such that, $\nu(t)(e_i)=t^{d_i}e_i$ for all $t\in K^\times$;
%\item[(3)] $\charac K$ does not divide $d_1^2+\cdots+d_n^2$.\footnote{This holds automatically if $\charac K=0$.}
\end{enumerate}
Then $H^2(G,K)\ni\bar{\alpha}(c_{\partial_1}\otimes c_{\partial_2})\neq 0$.
\end{proposition}

\begin{proof}
If the central extension of $G$ corresponding to the cocycle $c:=\alpha(c_{\partial_1}\otimes c_{\partial_2}):G\times G\to K$ splits, then it also splits over $H:=\nu(K^\times)$. Since $H$ is commutative, this implies $c(x,y)=c(y,x)$ for all $x,y\in H$. We will show that $c(\nu(s),\nu(t))\neq c(\nu(t),\nu(s))$ for some $s,t\in K^\times$, thus completing the proof.

Since $\partial_1$ and $\partial_2$ are linearly independent, there are $s,t\in K$ such that $\partial_1(s)\partial_2(t)-\partial_2(s)\partial_1(t)\neq 0$. In particular, $s,t\in K^\times$. 
By~\eqref{eq:alpha}, computing the trace with respect to the basis $\{e_i\}_{1\leq i\leq n}$, we obtain
$$
c(\nu(s),\nu(t))=\tr(c_{\partial_1}(\nu(s))c_{\partial_2}(\nu(t)))=\sum_{i=1}^n \partial_1(s^{d_i})s^{-d_i}\partial_2(t^{d_i})t^{-d_i}=(\sum_{i=1}^nd_i^2)\partial_1(s)\partial_2(t)(st)^{-1}.
$$  
It follows $c(\nu(s),\nu(t))\neq c(\nu(t),\nu(s))$.
\end{proof}

\subsection{Application to LDAGs}
\begin{corollary}\label{cor:Chev}
Let $H$ be a simple LAG over a universal $\Delta$-field $\U$ with $|\Delta|\geq 2$. There exists a non-splitting central extension
\begin{equation}\label{eq:ga}
\Ga\hookrightarrow E\stackrel{\pi}{\twoheadrightarrow} H
\end{equation}
in $\LDAG_k$ such that $E$ is perfect.   
\end{corollary}
\begin{proof}
Let $K:=\U$, $G:=H(\U)$, $V$ a nontrivial algebraic $G$-module and $\partial_1\neq\partial_2\in\Delta$. Then, by Proposition~\ref{prop:CupProduct} and basic properties of representations of simple LAGs (see, e.~g., \cite[Lemma 19, p.~27]{Steinberg}), the central extension of $G$ by $\U$ corresponding to the 2-cocycle $\alpha(c_{\partial_1}\otimes c_{\partial_2})$ is non-trivial. It is given by introducing the following group structure on the set $E':=G\times \U$:
$$
(g,x)(h,y):=(gh,x+y+\alpha(c_1\otimes c_2)(g,h))
$$
and definig $\pi':E'\to G$ to be the projection. Note that $E'=[E',E']$ (as an abstract group). Indeed, it suffices to show that $\U\subset [E',E']$ since $E'=\U\cdot [E',E']$. Let $\U^\times\subset G$ be a 1-dimensional algebraic torus. It follows from the proof of Proposition~\ref{prop:CupProduct} that, for all $s,t\in \U^\times$, the commutator of $s$ and $t$ in $E'$ equals
$$
[(s,0),(t,0)]=\left(1,c\frac{\partial_1(s)\partial_2(t)-\partial_1(t)\partial_2(s)}{st}\right),
$$
where $c\neq 0$ is a constant that does not depend on $s$ and $t$. Since $\U$ is differentially closed, the commutator above can take on any value in $\U$. Hence, $E'=[E',E']$. 

It remains to show that the central extension
\begin{equation}\label{eq:ga'}
\U\hookrightarrow E'\stackrel{\pi'}{\twoheadrightarrow} G
\end{equation}
we have constructed is obtained by taking $\U$-rational points of some extension~\eqref{eq:ga} in $\LDAG_{\U}$. This follows from the fact that all maps in~\eqref{eq:ga'} and the product map on $E'$ are differential polynomial (so, $E'$ is an LDAG). 
\end{proof}

\subsection{Linearity of the central extensions}\label{sec:linearity}
Let $K$, $G$ and $V$ be as in Section~\ref{sec:absconstr}. We will give another construction of a central extension $K\hookrightarrow E\twoheadrightarrow G$ depending on cocycles $c_1,c_2\in Z^1(G,K)$, and we will see that it is the one determined by $\alpha(c_1\otimes c_2)$. Moreover, it will be clear why $E$ is linear if $G$ is. 

To every $c\in H^1(G,\End(V))$ there correspond two $G$-module structures  $l_c,l'_c:G\to\GL(U)$ on the space $U:=\End(V)\oplus K$ given by 
\begin{align}
l_c(g)(A,a)&:=(gAg^{-1}+c(g)a,a),\\
l'_c(g)(B,b)&:=(gBg^{-1},b+\tr(c(g^{-1})B).
\end{align}
These are dual with respect to the pairing
$$
U\otimes_KU\to K,\qquad (A,a)\otimes(B,b)\mapsto \tr(AB)+ab.
$$
Let us consider the vector space $W:=K\oplus\End(V)\oplus K$ and the subgroup $P\subset\GL(W)$ of all transformations preserving the flag $K\subset K\oplus\End(V)\subset W$. Note the maps
\begin{align}
\pi:W\to U,\quad (b,A,a)\mapsto (A,a),\\
\iota:U\to W,\quad (B,b)\mapsto (b,B,0),
\end{align}
which induce homomorphisms $\pi_*, i^*:P\to\GL(U)$. 

Let $c_1,c_2\in Z^1(G,\End(V))$. Set
\begin{equation}\label{eq:linear}
E=E(c_1,c_2):=\{(g,p)\in G\times P\ :\ l_{c_1}(g)=\pi_*(p),\ l'_{c_2}(g)=\iota^*(p)\}.
\end{equation}
The projection $G\times P\to G$ induces the homomorphism
$$
\nu:E\to G.
$$
We have $\Ker\nu$ consisting of the elements $(1,p_t)$, where $t\in K$ and $p_t\in P$ is defined by 
$$
p_t(b,A,a)=(b+at,A,a),\qquad (b,A,a)\in W
$$
On the other hand, note the map of sets $$\sigma: G\to E,\qquad g\mapsto (g,p_g),\quad p_g(b,A,a):=\left(b+\tr(c_2(g^{-1})A),\; gAg^{-1}+c_{1}(g)a,\; a\right),$$
which is a section (on the set level) of $\nu$. Hence, $\nu$ is an epimorphism. Since $p_tp_g=p_gp_t$ for all $t\in K$, $g\in G$, $\Ker\nu\subset Z(E)$. Computation shows that, for all $g,h\in G$,
$$
p_{h}p_g=p_{hg}p_{t},\quad t=\alpha(c_1\otimes c_2)(h,g)\in K.
$$
Hence, the corresponding to $\nu$ class in $H^2(G,K)$ is represented by the cocycle $\alpha(c_1\otimes c_2)$. Note that, by definition of $E$, if $G\subset\GL(V')$, then $E\subset G\times P\subset\GL(V')\times\GL(W)\subset\GL(V'\oplus W)$.

\section*{Acknowledgements} 
I am grateful to Alexandru Buium, Phyllis Cassidy, James Freitag, James Humphreys, Alexey Ovchinnikov, Anand Pillay, Gopal Prasad and Michael Singer for very helpful discussions.

\bibliographystyle{spmpsci}

\bibliography{SimpleExt}

\end{document}